\documentclass{amsart}
\usepackage{amsfonts}
\usepackage[alphabetic]{amsrefs} 

\newtheorem{theorem}{Theorem}[section]
\newtheorem{lemma}[theorem]{Lemma}
\theoremstyle{definition}
\newtheorem{definition}[theorem]{Definition}
\newtheorem{example}[theorem]{Example}

\theoremstyle{remark}
\newtheorem{remark}[theorem]{Remark}
\numberwithin{equation}{section}
\theoremstyle{plain}

\newtheorem{corollary}{Corollary}

\newtheorem{proposition}{Proposition}

\newcommand{\gp}{\dot{\gamma}}

\copyrightinfo{2001}{enter name of copyright holder}
\begin{document}
\title {Some Curvature Pinching results for Riemannian manifolds with density.}

\author{William Wylie}
\address{215 Carnegie Building\\
Dept. of Math, Syracuse University\\
Syracuse, NY, 13244.}
\email{wwylie@syr.edu}
\urladdr{https://wwylie.expressions.syr.edu}


\begin{abstract}
 In this note we consider versions of both Ricci and sectional curvature pinching for Riemannian manifold with density.  
 %Many of the arguments have been used before in the study of gradient Ricci solitons.  However we take a slightly different approach, replacing Morse theory of the potential function $f$ with the theory of critical points of the distance function which allows us to handle the case where the potential field is non-gradient. 
In the Ricci curvature case  the main result implies a  diameter estimate that is new even for  compact shrinking Ricci solitons.  In the case of sectional curvature  we prove a new sphere theorem.   \end{abstract}

\maketitle

\section{Introduction}

%The general idea is to determine the optimal constant $\varepsilon$ such that the ratio of the  maximum and minimum of curvature is less than $\varepsilon$ then the geometry or topology is rigid in some way. In the weighted setting, the different  concepts of curvature gives somce flexibility in defining pinching, as we can take the ratio of the maximum of one kind of curvature with minimum of some other kind.  Thus, the problem becomes not only to determine the pinching constant but also to determine what kind of curvature should be taken as the numerator and denominator of the ratio. 

Let $(M,g)$ be a Riemannian manifold and $X$ be a smooth vector field on $M$.  There are a number of different ``weighted" curvatures defined for the triple $(M,g,X)$.  Perhaps the simplest is the Bakry-Emery Ricci tensor, which denote as $ \mathrm{Ric}_X  = \mathrm{Ric}+ \frac12 L_X g, $
 When $X = \nabla f$ for some function $f$, we write $\mathrm{Ric}_f = \mathrm{Ric} + \mathrm{Hess} f$.   
 
 The study of  $\mathrm{Ric}_X$  pre-dates Bakry-Emery \cite{BE} and  goes back at least to Lichnerowicz, who called it the $C$ tensor \cite{Lich1, Lich2}.   $\mathrm{Ric}_X$ appears in  Perelman \cite{Per} and Hamilton's  \cite{Hamilton} work on the Ricci flow and is related to  Lott-Vilanni \cite{LV} and Sturm's \cite{Sturm1, Sturm2} theory of Ricci curvature for metric measure spaces.  Partly due to  this motivation,  a number of mathematicians have recently investigated  the connections between  $\mathrm{Ric}_X$ and  the concepts of classical  Riemannian geometry.  There are far too many recent papers in this direction to reference all of them  in this note, so we will only cite the works \cite{Lott, Morgan, Morgan3, MunteanuWang, WeiWylie} as well as chapter 18  of \cite{MorganBook} and  the references there-in. 

The first pinching condition we will consider is  a positive lower bound on the Bakry-Emery Ricci curvature  and  an upper bound on the classical Ricci curvature.  After rescaling we will write the condition as
\begin{eqnarray} \label{RicPinch} \mathrm{Ric}_X \geq \varepsilon(n-1) g \qquad \mathrm{Ric} \leq(n-1)g \qquad \varepsilon > 0.  \end{eqnarray}
  First note that if $M$ is compact, then $\varepsilon \leq 1$ and $\varepsilon =1$ if and only if $M$ is an Einstein space and $X$ is a Killing field. This follows easily from the divergence theorem. % since if $\varepsilon \geq 1$ then $L_X g \geq 0$ and $\mathrm{tr}(L_X g) = \mathrm{div} X$, so $L_X g = 0$. 
   In the compact case we can thus think of the pinching constant $\varepsilon$ as measuring how far the space is from being an   Einstein manifold with a Killing field.

On the other hand, in the non-compact case it is  possible for $\varepsilon >1$, as the following example shows. 
   \begin{example} \label{Gaussian}  The Gaussian space  is the flat metric on $\mathbb{R}^n$ with  $f(x) = \frac12|x|^2$.  It  has $\mathrm{Ric}_f = g$ and $\mathrm{Ric}= 0$.  In fact,  any simply connected space with bounded  non-positive sectional curvature  has an $f$  such that $\mathrm{Ric}_f \geq g$.  See Example 2.2 of  \cite{WeiWylie}. 
\end{example} 

On the other hand, all of the manifolds in Example \ref{Gaussian} are diffeomorphic to $\mathbb{R}^n$.  Our first observation is that this is true in general. 

\begin{proposition}  \label{Prop_Euclidean}  Suppose $(M,g)$ is a complete Riemannian manifold and suppose that $\mathrm{Ric}_X > (n-1) g$ and  $ \mathrm{Ric} \leq(n-1)g$,   then $(M,g)$ is diffeomorphic to $\mathbb{R}^n$. 
\end{proposition} 

Note that Proposition \ref{Prop_Euclidean} is optimal as the product of an Einstein metric equipped with a Killing field  and a Gaussian will have $\mathrm{Ric}_X = (n-1)g$ and $\mathrm{Ric} \leq(n-1) g$.  We also note that there are  other interesting spaces satisfying (\ref{RicPinch}) as  all shrinking Ricci solitons ($\mathrm{Ric}_X = \lambda g$) with bounded Ricci curvature are included in this class.  Examples of shrinking Ricci solitons can be found in \cite{Cao, Koiso, FIK, WangZhu}. 

% Proposition \ref{Prop_Euclidean} shows that  interesting  topology for spaces satisfying (\ref{RicPinch}) only occur when $0 < \varepsilon \leq 1$.  

We are interested in the topology of spaces satisfying (\ref{RicPinch}).  
In this direction,  the author showed  in  \cite{WylieI} that if  $\mathrm{Ric}_X \geq \varepsilon(n-1) g$, then $M$ has finite fundamental group.  In  \cite{Fang} it was shown that if $X = \nabla f $ and (\ref{RicPinch}) holds then $M$ is topologically finite, i.e. it is homeomorphic to to the interior of a manifold with boundary.  Our first main result is to extend this second result to the case where $X$ is not a gradient field. 
\begin{theorem} \label{Thm_Topology}
Suppose $(M,g)$ is a complete Riemannian manifold  satisfying (\ref{RicPinch}), then $M$ is homeomorphic  to the interior of a manifold with boundary. 
\end{theorem}
  The new ingredient in our proof is to replace Morse theory  applied to the potential function $f$  in \cite{Fang} with the theory of critical points of the distance function.  The proof of Proposition  \ref{Prop_Euclidean} comes similarly  from considering critical points of the distance function.  
  
  We do not know whether the upper bound on Ricci curvature can be removed from Theorem \ref{Thm_Topology}.  In \cite{Fang} it is also shown that a shrinking gradient Ricci soliton is topologically finite if the scalar curvature is bounded. 
  
 Myers' diameter estimate states that a complete  space with a positive lower bound on the  Ricci tensor has diameter less than or equal to the diameter of the sphere of corresponding constant curvature.   %which states that a complete manifold with $\mathrm{Ric} \geq (n-1) \varepsilon g$ has diameter less than or equal to $\frac{\pi}{\sqrt{\varepsilon}}$.  
 Non-compact examples like Example \ref{Gaussian} show Myers' theorem is  not  true for the spaces satisfying (\ref{RicPinch}).   On the other hand, we prove the following gap theorem for the diameter of compact spaces satisfying (\ref{RicPinch}).   
 \begin{theorem}\label{Thm:Diam}
 Suppose that $(M,g)$ is compact  satisfying (\ref{RicPinch}), then for any $p \in M$,  $M \subset B\left( p,   \frac{(n-1) \pi + |X(p)|}{(n-1) \varepsilon}\right)$.  In particular, if $X$ has a zero then $\mathrm{diam}(M) \leq \frac{2 \pi}{\varepsilon}$. 
 \end{theorem}
%\begin{theorem}  \label{Thm:Diam} Suppose that $(M,g)$ is a compact manifold supporting a vector field $X$ which satisfies (\ref{RicPinch})  and has a zero,  then $\mathrm{diam}(M) \leq \frac{2 \pi}{\varepsilon}$. 
%\end{theorem}

We also obtain a gap theorem for the injectivity radius. 

\begin{theorem} \label{Thm:Inj}Suppose that $(M,g)$ is complete manifold supporting a vector field $X$ which satisfies (\ref{RicPinch}). If there is a point $p\in M$ such that $\mathrm{inj}_p(M) \geq  \frac{(n-1) \pi + |X(p)|}{(n-1) \varepsilon}$ then $M$ is diffeomorphic to $\mathbb{R}^n$.  In particular, if $X$ has a zero and $\mathrm{inj}(M) \geq \frac{\pi}{\varepsilon}$, then $M$ is diffeomorphic to $\mathbb{R}^n$. 

\end{theorem}

We consider these results  gap theorems because Example \ref{Gaussian} shows that there are spaces with  infinite diameter and injectivity radius such that for every $\varepsilon$ there is  a vector field $X$ with a zero satisfying (\ref{RicPinch}).  Theorems \ref{Thm:Diam} and \ref{Thm:Inj} show that,  for $\varepsilon$ fixed, there can not be compact spaces satiisfying (\ref{RicPinch}) with arbitrarily large diameter or injectivity radius. 

%\begin{remark}  In section 3 we construct examples of manifolds with density on the sphere  satisfying (\ref{RicPinch}) with diameter  and injectivity radius arbitrary close to $\pi/\varepsilon$.  In particular these example have larger diameter than, $\pi/\sqrt{\varepsilon}$.  Moreover, rescaling these  examples gives compact examples with $\mathrm{Ric}_X \geq (n-1)$ with arbitrarily large diameter and injectivity radius, showing the upper bound on curvature in Theorems \ref{Thm:Diam} and \ref{Thm:Inj}  is  necessary.  
%\end{remark} 

$X$ will always have a zero when it is the gradient of some function $f$, see Corollary \ref{Cor:F}.   On the other hand, an odd dimensional sphere admits a non-trivial constant length  Killing field,  giving an example satisfying (\ref{RicPinch}) where $X$ does not have a zero. 

When applied to a  gradient shrinking Ricci soliton, these results give  bounds on the diameter and injectivity radius in terms of an upper bound on the Ricci curvature.  Since this appears to be a new result, we state it as a corollary. 

\begin{corollary} If $(M,g, f)$ is a complete gradient shrinking Ricci soliton,  $\mathrm{Ric}_f = \lambda g$,  with $\mathrm{Ric} \leq \rho$,  then either $\mathrm{inj}(M) < \frac{\pi \rho}{\lambda}$ or $M$ is diffeomorphic to $\mathbb{R}^n$.  If $M$ is,  in addition,  compact then $\mathrm{diam}(M) \leq \frac{2 \rho \pi}{\lambda}$.
\end{corollary}
Recently Munteanu and Wang have also established an upper bound on the diameter in terms of the injectivity radius of a gradient Ricci soliton \cite{MunteanuWang2}. Proposition 2.5 of  \cite{Weber} also  gives an upper bound on the diameter of a compact  Ricci soliton in terms of an upper bound on the scalar curvature along with a bound on the constant $C_1 = |\nabla f|^2 + 2 \lambda f + \mathrm{scal}$ where $f$ is the assumed to be normalized so that $\int_M f = 0$.    There is also a universal lower bound on the diameter of a non-Einstein shrinking gradient Ricci soliton \cite{Futaki}. 

 In order to obtain finer topological information about curvature pinching, we consider a corresponding version of sectional curvature pinching.  While weighted Ricci curvature has been developed extensively, concepts of weighted sectional curvature have been far less studied.     In \cite{WylieII} the author introduced two notions of weighted sectional curvature and they have been subsequently studied in \cite{KennardWylie}.  The notion of weighted sectional curvature we will use in this paper is the following. 
 
 \begin{definition} Suppose $(M,g)$ is a Riemannian manifold and $X$ is a smooth vector field on $M$,  we say that $(M,g,X)$ has a weighted sectional curvature lower bound $\varepsilon$ and write $\sec_X \geq \varepsilon$ if 
  for every two-plane $\sigma$ in $T_pM$ we have 
\[ \mathrm{sec}(\sigma) \geq \varepsilon - \frac12 L_X(U,U) \qquad \forall U \in \sigma \quad  |U|=1 \]
\end{definition}

Note that  if $\sec_X \geq \varepsilon$, then $\mathrm{Ric}_{(n-1)X} \geq (n-1) \varepsilon$. Moreover, from the discussion above  we see that if a compact manifold satisfies $\sec_X \geq \varepsilon$ and $\sec \leq 1$ then $\varepsilon \leq 1$ and $\varepsilon = 1$ if and only if $(M,g)$ has constant curvature and $X$ is a Killing field and that a complete space with  $\sec_X > 1$ and  $\sec \leq 1$ is diffeomorphic to $\mathbb{R}^n$.  For examples of metrics with $\sec_X > 0$ but with some negative curvatures, see \cite{KennardWylie}. 

 The most famous pinching theorem in Riemannian geometry is the quarter-pinched sphere theorem which states that a  simply connected manifold with $1/4 < \sec \leq 1$  is diffeomorphic  to the sphere.  We are interested in whether there is a similar pinching phenomenon for simply connected manifolds satisfying the condition $\sec_X \geq \varepsilon$ and $\sec \leq 1$.  Note that this class of spaces includes the sphere and Euclidean space for all $\varepsilon \leq 1$.   We  prove a classification up to homotopy equivalence for $\varepsilon > 1/2$.       

\begin{theorem} \label{Thm:Sec} Suppose that $(M,g)$ is a simply connected Riemannian manifold and $X$ is a vector field which admits a zero and so that $\sec_X \geq \varepsilon > 1/2$ and $\mathrm{sec} \leq 1$.
\begin{enumerate}
\item If $M$ is compact then $M$ is homeomorphic to a sphere. 
\item If $M$ is complete and non-compact then $M$ is contractible. 
\end{enumerate}
\end{theorem}

The quarter-pinched sphere theorem has a long  history.  The homeomorphism classification goes back to the 60s and is due to Berger \cite{Berger2}  and Klingenberg \cite{Kling} while the diffeomorphism classification was established using Ricci flow techniques recently by Brendle and Schoen \cite{BrendleSchoen}.   The proof of Theorem \ref{Thm:Sec} follows from adapting the classical arguments  of Berger and Klingenberg.  In fact, the only tool we use is Morse theory for the energy functional, as we do not even have triangle comparison results a la Alexandrov for $\sec_X$.     Because of this,  we can only show classification up to homotopy equivalence.  In the compact case, the resolution of the Poincare conjecture then gives the homeomorphism classification.  On the other hand, there are contractible manifolds which are not homeomorphic to $\mathbb{R}^n$, so a homeomorphism classification can not follow from topological arguments in the non-compact case.    We expect, particularly in the noncompact case, that Theorem \ref{Thm:Sec} can be improved with further study of the weighted sectional curvature $\sec_X$.  

%We would like to emphasize that the results in this paper are  novel in two ways when compared to other  results relating weighted curvatures and topology such as  \cite{WeiWylie},  \cite{MunteanuWang}, or \cite{WylieII}. Firstly, our results hold for any  field $X$ which has a zero  and not just gradient fields.  Secondly, in the gradient case $X = \nabla f$ we do not assume any bounds on the growth of the potential function $f$ and instead assume an upper curvature bound.  

The paper is organized as follows.  In the next section we prove the results about Ricci pinching and review the concept of critical points of the distance function.  In section three we discuss the examples showing the results from section 2 are optimal.  In section 4 we discuss Theorem \ref{Thm:Sec}.  
 
 \section{Ricci pinching and critical points of the distance function}
 
The starting point for all of  results of this note is the following estimate  involving the second variation of energy formula.  It is by no means new and similar results were used, for example, by Hamilton in studying the change in the Riemannian distance along   the Ricci flow (See Section 17 of \cite{Hamilton}). Also see  Proposition 1.94 of \cite{Chowetc}.

\begin{lemma} \label{IntRicci}
Suppose $(M,g)$ is a Riemannian manifold  and $\gamma:[0,r] \rightarrow M$ is a unit speed minimizing geodesic  of length greater than or equal to $\pi$.  Suppose that  $\mathrm{Ric} \leq (n-1)g$ on $B\left(\gamma(0), \frac{\pi}{2}\right)$ and $B\left(\gamma(r), \frac{\pi}{2}\right)$   Then 
\[ \int_0^r \mathrm{Ric}(\gp, \gp) dt \leq (n-1)\pi \]
\end{lemma}

\begin{proof}
From the second variation or arclength formula, for a minimal geodesic we have 
\[ 0 \leq \int_0^r (n-1)(\phi')^2 - \phi^2 \mathrm{Ric}(\gp, \gp) dt \]
where $\phi$ is any function with $\phi(0) = \phi(r) = 0$.     Which implies that 
\[  \int_0^r \mathrm{Ric}(\gp, \gp) dt \leq \int_0^r (n-1)(\phi')^2+ (1 - \phi^2) \mathrm{Ric}(\gp, \gp) dt \]
%So if we assume that $|\phi|\leq 1$ have
%\[(n-1)\int_0^r 1+ (\phi')^2 - \phi^2  dt  \]
Choose the function 
\[ \phi(t) = \left \{ \begin{array}{cc} \sin(t)  & 0 \leq t \leq \frac{\pi}{2} \\ \\1 & \frac{\pi}{2} \leq t \leq r - \frac{\pi}{2} \\ \\ -\sin(t-r)  & r- \frac{\pi}{2}\leq t \leq r \end{array} \right. \]
Then elementary calculation yields
 \begin{eqnarray*}
\int_0^r \mathrm{Ric}(\gp, \gp) dt   &\leq& (n-1)\left(  \int_0^{\pi/2}  1+ (\phi')^2 - \phi^2  dt   +  \int_{r-\frac{\pi}{2}}^{r}  1+ (\phi')^2 - \phi^2 dt\right) \\
  Ê &=& (n-1) \left(\int_0^{\pi/2} 1 + \cos^2(t) - \sin^2(t) dt + \int_{r-\frac{\pi}{2}}^{r} 1 + \cos^2(t-r) - \sin^2(t-r) dt\right)\\
&=& (n-1) \pi
\end{eqnarray*}
 \end{proof}
 
Under our Ricci pinching assumption this gives us the following estimate. 
 
 \begin{lemma} \label{Lem:Crit}  Let $(M,g)$ be a complete Riemannian manifold satisfying  (\ref{RicPinch}). Suppose that  $p,q \in M$ and let $\gamma$ be a unit speed minimal geodesic from $p$ to $q$, then 
 \[ g(X(q), \dot{\gamma}(d(p,q))) \geq -(n-1) \pi -|X(p)| + (n-1) \varepsilon d(p,q) \]
 \end{lemma}
 
 \begin{proof}
 We have 
 \[ \frac12 L_X g(\dot{\gamma}, \dot{\gamma}) = g( \nabla_{\dot{\gamma}} X, \dot{\gamma} ) = \frac{d}{dt} g(X, \dot{\gamma}) \]
 where $\frac{d}{dt}$ denote derivative along the geodesic. If we integrate the equation $\mathrm{Ric} + \frac12 L_X g \geq  \varepsilon(n-1) g$ along the geodesic we obtain, 
 \[ \int_0^{d(p,q)} \mathrm{Ric}(\dot{\gamma}, \dot{\gamma}) dt  + g(X(q), \dot{\gamma}(d(p,q))) - g(X(p), \dot{\gamma}(0)) \geq (n-1)\varepsilon d(p,q) \]
 Applying Lemma \ref{IntRicci} then gives us the formula. 
 \[ g(X(q), \dot{\gamma}(d(p,q))) \geq -(n-1) \pi -|X(p)| + (n-1) \varepsilon d(p,q) \]
 \end{proof}
 
 Two corollaries of this formula that we will find useful are the following.  The second is already stated in \cite{Fang} and also appears implicitly as part of the proof of  Lemma 1.2 of \cite{Per2}.  We include  the proofs for completeness. 
 
 \begin{corollary} \label{Cor:X} If $(M,g)$ is compete and non-compact and satisfies (\ref{RicPinch}) then $|X| \rightarrow \infty$ at infinity. \end{corollary}
 
 \begin{proof}
 Fix $p$ and let $q \rightarrow \infty$, then we obtain that $g(X(q), \dot{\gamma}(d(p,q))) \rightarrow \infty$ so that in particular, $|X| \rightarrow \infty$. 
 \end{proof}
 
 \begin{corollary} \label{Cor:F} If $(M,g)$ is compete and non-compact satisfying  (\ref{RicPinch}) then $f$ grows at least quadratically with the distance to any point.  In particular, $f$ has at least one critical point as it obtains its minimum. 
 \end{corollary}
 \begin{proof}
 When $X = \nabla f$, for any geodesic $\gamma(t)$ we have 
 \begin{eqnarray*}
 \frac{d}{dt} \left( f(\gamma(t)) \right)  &=& g(X(\gamma(t)), \dot{\gamma}(t)) \\
 &\geq& -(n-1) \pi -|X(\gamma(0))| + (n-1) \varepsilon t
 \end{eqnarray*}
 Integrating the equation along $\gamma$ then gives the result. 
 \end{proof}
 
 \begin{remark} In the case of a gradient Ricci soliton  Cao-Zhou  show that $f$ grows exactly quadratically \cite{Cao-Zhou}.  \end{remark}

In order to obtain topological results from these formulae we will use the theory of critical points of the distance function which was pioneered by Grove-Shiohama in  \cite{GS}.  There are many surveys of the theory.  For completeness  we will review the main elements and the refer to reader, for example,  to Chapter 11 of \cite{Petersen} for a more complete treatment. 

  Recall that for a smooth function $f:M \rightarrow \mathbb{R}$, the critical points are the points where $df = 0$.  If $f$ is  proper and  there are no critical points in $f^{-1}([a,b])$ one of the foundational lemmas of Morse theory says that  $f^{-1}((-\infty, b])$ deformation  retracts onto  $f^{-1}((-\infty, a])$.  Thus, in the case $X = \nabla f$,  Lemma \ref{Lem:Crit} already shows that $f$ is proper and has no critical points outside of some compact set.  

Let $p \in M$ and consider the distance function $r$ to $p$. $r$  is smooth almost everywhere with unit gradient so it does not have critical points in the traditional sense. The approach to defining  critical points for $r$  is to view critical points as  the obstruction to producing a deformation retraction between sub-levels.      At the points where $r$ is smooth, $\nabla r$ is the unit tangent vector to the unique  minimal geodesic from $p$ to the point.   At the points where $r$ is not smooth, there may be multiple minimal geodesics  from $p$ to the point.  Intuitively, it still might be possible to deformation retract past these points if all these geodesics point in roughly the same direction.   Indeed,  it turns out all we need to build the retract is that all of  the tangent vectors of the  minimal geodesics from $p$ to $q$ to  lie in some  half space of $T_qM$.   This gives the following definition of a critical point. 

 \begin{definition} Fix $p\in M$, a point $q$ is a  critical point of the distance function to $p$ (is critical to $p$) if for every vector $V \in T_q M$ there is a minimal geodesic $\gamma$ with $\gamma(0) = p$, $\gamma(d(p,q)) = q$ such that  $g(\dot{\gamma}(d(p,q)), V) \leq0$. 
 \end{definition}
 
 In all of the results of this paper, we'll show that certain points $q$ are not critical to $p$ by showing that for every minimal geodesic from $p$ to $q$, $g(\dot{\gamma}(d(p,q)), X(q)) > 0$.  In particular, our arguments to do not have an analogue  in the classical case where $X = 0$. 
 
 As alluded to in the discussion above, we have the following topological lemma about critical points of the distance function. See page 337 of \cite{Petersen}. 
 
 \begin{lemma} Suppose that there are no critical points of the distance function to $p$ in the annulus $ \{ q : a \leq d(p,q) \leq b \}$. 
 Then $B(p, a)$ is homeomorphic to $B(p,b)$ and $B(p,b)$ deformation retracts onto $B(p,a)$.  Moreover, if there are no critical points of $p$ in $M$ then $M$ is diffeomorphic to $\mathbb{R}^n.$
 \end{lemma}
 
\begin{remark} We can not get diffeomorphic sub-levels in general because the levels of $r$ will not in general be smooth.  This is not an issue when there are no critical values because we can retract down to a small neighborhood where the levels are smooth. 
\end{remark}

Our first application is Proposition \ref{Prop_Euclidean}. 

\begin{proof} [Proof of Proposition \ref{Prop_Euclidean}]
 From the assumptions we have  $L_X g >0$. From  the divergence theorem this implies the manifold is non-compact.  From Corollary \ref{Cor:X} we have $|X|^2 \rightarrow \infty$ at infinity which  implies that the function $\phi(p)= |X(p)|^2$ has a minimum somewhere on $M$.  At the minimum point we must have 
 \[ 0 = D_X \phi = D_X g(X,X) = 2 L_Xg(X,X), \]
 since $L_X g > 0$ this implies the minimum  of $\phi$ must be zero.  Let $\gamma(t)$ be a geodesic with $\gamma(0)=p$ a zero of $X$, then we obtain 
 \begin{eqnarray*}
 0 < \int_0^t L_X g(\dot{\gamma}, \dot{\gamma}) dt= g(X(\gamma(t)), \dot{\gamma}(t)) 
 \end{eqnarray*}
 which shows that there are no critical points to $p$ and thus the manifold is diffeomorphic to Euclidean space. 
\end{proof}

\begin{remark} We note that the last part of the proof of Proposition \ref{Prop_Euclidean} proves the following: \emph{If $M$ is a complete manifold supporting a vector field with a zero such that $L_X g >0$ then $M$ is diffeomorphic to Euclidean space.} We note that even in the gradient case this is slightly different than the result one would obtain from classical Morse theory as we do not need to  assume that   $f$ is proper but have added  the assumption that the  metric is complete.  On the other hand, the assumption that $X$ has a zero is necessary.  A simple example is to let  $g = dt^2 + e^{2t} g_N$ be the warped product metric on $\mathbb{R} \times N$ and define $f= e^t$,  which has  $\mathrm{Hess}_g f > 0$ . 
\end{remark}

More generally,  we have the following critical point estimate for metrics with pinched Ricci curvatures. 
\begin{lemma} \label{CritPoint} 
Let $(M,g)$ be a complete Riemannian manifold satisfying (\ref{RicPinch}).  Let $p \in M$, then if $q \in M$ with 
\[ d(p,q) >  \frac{(n-1) \pi + |X(p)|}{(n-1) \varepsilon}  \]
Then $q$ is not critical to $p$. 
\end{lemma}

 \begin{proof} 
If 
 \[ d(p,q)> \frac{(n-1) \pi + |X(p)|}{(n-1) \varepsilon}  \]
 then by Lemma \ref{Lem:Crit},   $g(X(q), \dot{\gamma}(d(p,q))) >0$. Since this is true for all minimal geodesics $\gamma$ from $p$ to $q$, this shows that $q$ is not critical to $p$. 
\end{proof}

This immediately gives us Theorem  \ref{Thm_Topology}. 

\begin{proof} [Proof of Theorem  \ref{Thm_Topology}] 
Fix $p$.  Lemma \ref{CritPoint} tells us that there are no critical points to $p$ outside of a compact set.  Thus the manifold deformation retracts onto  a finite  open ball in $M$, showing it is homeomorphic to the interior of a manifold with boundary. 
\end{proof}

Now we consider the diameter estimate.  Unlike the results above, this result does not seem to have been observed before even in the gradient case.   The main ``new" ingredient is the following fact about critical points of the distance function due to Berger \cite{Berger2}. 
 
 \begin{lemma} \cite{Berger2} \label{BerCrit}  Let $(M,g)$ be a compact Riemannian manifold and fix $p\in M$.  Let $q$ be a point such that $d(p,q) = \max_{x \in M} d(p,x)$.  Then  $q$ is critical to $p$. 
 \end{lemma}
 
\begin{remark} This result is often stated for $p$ and $q$ which realize $\mathrm{diam}(M,g) = d(p,q)$, but the same proof  works in this case. See \cite{doC} p. 283. \end{remark}
The proofs of Theorems \ref{Thm:Diam} and \ref{Thm:Inj} are now immediate from Berger's Lemma and Lemma \ref{CritPoint}. 

 \begin{proof}[Proof of Theorem \ref{Thm:Diam}]
 By Lemma \ref{BerCrit} if $q$ is a  furthest point from  $p$ then $q$  must be critical to $p$.  Then by Lemma \ref{CritPoint}, $d(p,q) \leq \frac{(n-1) \pi + |X(p)|}{(n-1) \varepsilon}$. 
 \end{proof}

  \begin{proof}[Proof of Theorem \ref{Thm:Inj}]
From  the hypothesis on injectivity radius  there are no critical points of the distance function inside  $B\left( p,   \frac{(n-1) \pi + |X(p)|}{(n-1) \varepsilon}\right)$. Lemma \ref{CritPoint} tells us there are also no critical points outside of the ball so  the manifold is diffeomorphic to $\mathbb{R}^n$. 
  \end{proof}
  
%  The first  systole $\mathrm{sys}_1(M)$ of a Riemannian manifold is half the length of the shortest non contractible loop.  A perhaps less interesting gap theorem which follows from the same argument is that if a metric space satisfies $\mathrm{sys}_1(M) >  \frac{\pi}{\varepsilon}$ then the manifold is simply connected.   NOTE: I think this can be made into $ \frac{\pi}{2\varepsilon}$
%  

\section{Example}

In this section we construct examples showing  the  gap theorems in the previous section are optimal.  In particular we have the following proposition. 

\begin{proposition} Let $n \geq 3$.  For any $\varepsilon < \frac{n-1}{n-2}$ there is a family of Riemannian manifolds with density on the $n$-sphere $(g_{\delta} ,f_{\delta})$ defined for $\delta$ sufficiently small  with $Ric_{g_{\delta}, f_{\delta}} \geq (n-1) \varepsilon$, $\mathrm{Ric}_{g_{\delta}} \leq (n-1)$ such that there is a point $p$ with $\sup_{q \in S^n} d_{g_\delta}(p,q) =  \mathrm{inj}_p(g_{\delta}) = L_{\delta}$ where  $L_{\delta} \rightarrow \frac{\pi}{\varepsilon}$ as $\delta \rightarrow 0$. 
\end{proposition}

\begin{proof} 
 The metrics are rotationally symmetric metrics of the form 
\[ dr^2 + \phi^2(r) g_{S^{n-1}}. \]
Let $\partial_r$ denote the tangent vector in the $r$ direction and let $X$ be a unit vector perpendicular to $\partial_r$.   We will define $f = f(r)$ our potential function to be a function of $r$, then the Bakry-Emery  tensor has the formula 
\begin{eqnarray*}
\mathrm{Ric}_f (\partial_r, \partial_r) &=& -(n-1) \frac{\ddot{\phi}}{\phi} + \ddot{f} \\
\mathrm{Ric}_f (\partial_r, X) &=& = 0 \\
\mathrm{Ric}_f(X,X) &=& - \frac{\ddot{\phi}}{\phi} + (n-2) \left( \frac{1 - (\dot{\phi})^2}{\phi^2} \right) + \frac{\dot{f} \dot{\phi}}{\phi} 
\end{eqnarray*}

Consider a $C^2$ half-capped cylinder given by $\phi$ defined on $[0,L]$  as 
\[ \phi(r) = \left \{ \begin{array}{cc} \sin(r) & 0 \leq r < \frac{\pi}{2} - \delta \\ \phi_0(r) & \frac{\pi}{2} - \delta \leq r < \frac{\pi}{2}  \\ A & \frac{\pi}{2} \leq r  \leq L \end{array} \right. \]
where  $\phi_0(r)$ is will be described later, and $L$, $\delta$ and $A$  are positive numbers to be chosen. 
For our potential function we define 
\[ f(r) = \left\{\begin{array}{cc}  -\frac{n-1}{2} (1-\varepsilon)r^2 &   0 \leq r < \frac{\pi}{2} - 2\delta\\ f_0(r) &  \frac{\pi}{2} - 2\delta \leq r < \frac{\pi}{2} - \delta \\ \frac{(n-1)\varepsilon}{2}( r - \frac{\pi}{2} + \delta)^2 - (n-1)(1 - \varepsilon)(\frac{\pi}{2} - 2 \delta) (r - \frac{\pi}{2} + \delta) + B&  \frac{\pi}{2} - \delta \leq r \leq L \end{array} \right. \]
Where the function $f_0$ will be described later and $B$ is another  constant to be chosen. 

 For $r >  \frac{\pi}{2} - \delta$, $f$ is a parabola with a critical point at the point $ \frac{\pi}{2\varepsilon} + \delta + 2 \frac{\delta}{\varepsilon}$ and the metric is a flat cylinder.  Thus we can double  $f$ and $\phi$  to obtain a smooth metric on the sphere of diameter and injectivity radius, $L_{\delta} = \frac{\pi}{\varepsilon} + 2\delta + 4 \frac{\delta}{\varepsilon}$.

For $r \in [0, \frac{\pi}{2} - 2 \delta]$ we have 
\begin{eqnarray*}
\mathrm{Ric}_f (\partial_r, \partial_r) &=& (n-1) \varepsilon \\
\mathrm{Ric}_f (X,X) &=& (n-1) \left( 1 - (1-\varepsilon) r \cot(r) \right) \\
&\geq& (n-1) \varepsilon 
\end{eqnarray*}
Since $r\cot(r)$ decreases from $1$ to zero on $[0,\frac\pi2]$. Also, for $r \geq  \frac{\pi}{2} $ we have
\begin{eqnarray*}
\mathrm{Ric}_f (\partial_r, \partial_r) &=& (n-1) \varepsilon \\
\mathrm{Ric}_f (X,X) &=& (n-2) (1/A)
\end{eqnarray*}
Thus as long as $\varepsilon < \frac{n-2}{(n-1)A}$, we have $\mathrm{Ric}_f \geq (n-1) \varepsilon g$ outside of region $ \frac{\pi}{2} - 2\delta \leq r < \frac{\pi}{2}$.  
To control the curvature in that region we must  choose  $\phi_0$ and $f_0$ appropriately.  In order for $f$ to be smooth at $\pi/2 - 2 \delta$  we choose $f_0$ to  satisfy the initial conditions
\[  \ddot{f_0}(\frac{\pi}{2} - 2 \delta) = -(n-1)(1- \varepsilon) \quad \& \quad \dot{f_0}((\frac{\pi}{2} - 2 \delta) = -(n-1) (1-\varepsilon) (\frac{\pi}{2} - 2 \delta).  \]
Define $f_0$ by prescribing $\ddot{f}_0(r)$ to be a monotone increasing function that satisfies  $\ddot{f}_0(\pi/2 - \delta) = (n-1)\varepsilon$ and $\int_{\pi/2 - 2\delta}^{\pi/2 - \delta} \ddot{f}_0(t) dt = 0$.  The second condition assures that  $f$ will also be smooth at $\pi/2 - \delta$ after choosing $C$ to be an appropriate constant. 

Then for any $r \in [\pi/2 - 2 \delta, \pi/2-\delta]$ we have  
\begin{eqnarray*}  \dot{f}_0(r)& =&  \int_{\pi/2 - 2\delta}^r \ddot{f}(t) dt + \dot{f}_0(\pi/2 - 2\delta) \\
&\geq &- (n-1) (1-\varepsilon) (\frac{\pi}{2} -  \delta)
\end{eqnarray*}
Then for the curvature we have 
\begin{eqnarray*}
\mathrm{Ric}_f (\partial_r, \partial_r) &\geq& (n-1)\varepsilon \\
\mathrm{Ric}_f(X,X) &=& (n-1)+ \dot{f}(r) \cot(r)\\
&\geq & (n-1) \left(  1 - (1 - \varepsilon)(\pi/2 - \delta) \cot(r) \right)
\end{eqnarray*}
Since $\cot(r) \rightarrow 0$ at $\pi/2$ we can thus see that $\mathrm{Ric}_f(X,X) \geq (n-1) \varepsilon$ if $\delta$ is chosen small enough. 

Define $\phi_0$ similarly by assuming it satisfies the initial conditions 
\[ \phi_0 (\frac{\pi}{2} - \delta) = \sin (\frac{\pi}{2} - \delta) \quad \& \quad  \dot{\phi}_0((\frac{\pi}{2} - \delta) = \cos(\frac{\pi}{2} - \delta)\quad \& \quad  \ddot{\phi}_0((\frac{\pi}{2} - \delta) = -\sin (\frac{\pi}{2} - \delta)\]
And letting $\ddot{\phi}$ increase from $-\sin(\frac{\pi}{2} - \delta)$ to $0$ such that it satisfies $  \cos(\pi/2 - \delta) = \int_{\pi/2 - \delta}^{\pi/2} \ddot{\phi}(t) dt $. Then choosing $A$ to be an appropriate constant we will have a $C^2$ function $\phi$.   When $\delta$ is small, $A$ will be close to $1$ and we  can see  that $\phi_0$ will be close to $A$,  $\dot{\phi_0}$ is close to zero and $\ddot{\phi} \geq 0$, which implies that $\mathrm{Ric}_f \geq (n-1) \varepsilon$ in this region as well.  
\end{proof}

 \section{Pinched sectional curvature}
  
 Now we fix some more  notation for the weighted sectional curvature.  Given an orthonormal pair of vectors  $U, V$  we define 
\[ \sec_X (U,V) = \sec(U,V) + \frac{1}{2} L_Xg(U,U) \]
Where $\sec(U,V)$ is the sectional curvature of the plan spanned by $U$ and $V$.  Note that the weighted curvature $\sec_X(U,V)$ is not symmetric in $U$ and $V$.  See Section 2 of \cite{WylieII}  for further discussion. 

As is traditional, the topological tool  we will  use here is not critical points of the distance function, but the Morse theory of the   energy functional applied to the path space.  For submanifolds $A$ and $B$ in $M$ define the path space as  
 \[ \Omega_{A,B}(M) = \{ \gamma:[0,1] \rightarrow M, \gamma(0) = A, \gamma(1) = B \}
 \]  
 We will only look at the space of paths between points $p$ and $q$,   $\Omega_{p,q}$.  We consider the Energy $E:  \Omega_{p,q}(M) \rightarrow \mathbb{R}$ and  variation fields  which vanish at both end points of the curve. The  critical points of $E$ are then the geodesics connecting $p$ and $q$ and we say  the index of such a geodesic is $\geq k$ if there is a $k$-dimensional space of variation fields along the geodesic which have negative second variation.    In order to estimate the index of a geodesic, we have the following lemma for sectional curvature modeled on Lemma   \ref{IntRicci}. 
  
 \begin{lemma} \label{Lem:Index} Suppose $(M,g)$ is a Riemannian manifold  with $\mathrm{sec} \leq 1$ and $\gamma:[0,r] \rightarrow M$ is a unit speed minimizing geodesic of length greater than or equal to $\pi$.   Suppose that for any perpendicular parallel field along $\gamma$, $E$, we have  
\[ \int_0^r \mathrm{sec}(\gp, E) dt > \pi \]
then  the index of $\gamma$ is greater than or equal to $(n-1)$.  \end{lemma}
 
 \begin{proof} 
   Let $V$ be the proper variation field $\phi E$ where $\phi$ is the function constructed in Lemma \ref{IntRicci}. 
   The second variation of arc length formula shows that 
   \[ \frac{d^2}{ds^2} E(0) = \int_0^r  (\phi')^2 - \phi^2 \mathrm{sec}(\gp, E) dt  \]
   Assume that  $ \frac{d^2}{ds^2} E(0) \geq 0$, then imitating  the proof of  Lemma \ref{IntRicci}  implies that $ \int_{\gamma} \mathrm{sec}(\gp,E) dt \leq \pi$.
 Therefore if  $ \int_0^r \mathrm{sec}(\gp, E) dt >\pi$ then the second variation must be negative, since there are $(n-1)$ linearly independent fields, we obtain the result.  
 \end{proof}

Applying the Lemma to the elements of $\Omega_{p,p}$ gives us the following. 
 
 \begin{lemma} \label{Lem:Index2}
 Suppose that $M$ is a complete Riemannian manifold supporting a vector field $X$  with a zero at a point $p$ such that
    \[  \mathrm{sec}_X\geq  \varepsilon \text{ and }\mathrm{sec} \leq 1 \qquad 0<\varepsilon<1. \] If $\gamma$ is a geodesic  loop  based at $p$  
    of length greater than $\frac{\pi}{\varepsilon}$  then $\mathrm{index}(\gamma) \geq(n-1)$. 
  \end{lemma} 
  
  \begin{proof}
  Let $E$ be a unit parallel field around $\gamma$, then 
  
 \begin{eqnarray*}
   \varepsilon \mathrm{length}(\gamma)\leq  \int_0^r \mathrm{sec}_X(\gp, E) dt  &=& \int_0^r \mathrm{sec}(\gp, E) + \frac{d}{dt} g(X, \gp)  dt  \\
   &=&  \int_0^r \mathrm{sec}(\gp, E) dt
   \end{eqnarray*}
   Therefore, for such a geodesic, $  \int_0^r \mathrm{sec}(\gp, E) dt> \pi$ for all unit perpendicular parallel fields.   By the previous lemma, $\mathrm{index}(\gamma) \geq(n-1)$.
   \end{proof}
   
   \begin{remark} In the  example in the previous section, if we let $\gamma$ be the geodesic in the $\partial_r$ directions we have $\sec_f(\partial r, E) \geq \varepsilon$, $\sec \leq 1$, and we can make $\gamma$ minimizing for a length arbitrarily close to $\frac{\pi}{\varepsilon}$ by taking $\delta \rightarrow 0$.  This shows the proof of Lemma \ref{Lem:Index2} is optimal.  On the other hand, the examples of section 2 do not have $\sec_X >0$ because if $U,V$ are both perpendicular to $\partial_r$ on the region where $r> \pi/2$, $\sec_X(U,V) = 0$.  For examples of compact  metrics with $\sec_X >0$ and $\sec\leq 1$ which do not have $\sec>0$ we refer the reader to \cite{KennardWylie}. 
   \end{remark}
   
   This gives us the following generalization of a sphere theorem of Berger \cite{Berger1}. 
   
   \begin{theorem} \label{Thm:Berg}
   Suppose that $M$ is a complete Riemannian manifold supporting a vector field $X$  with a zero at a point $p$ such that
    \[  \mathrm{sec}_X > \varepsilon \text{ and }\mathrm{sec} \leq 1 \qquad 0<\varepsilon<1. \] 
    If $\mathrm{inj}_p(M) > \frac{\pi}{2 \varepsilon}$ then $M$ is either homotopic to a sphere or it is contractible.  
    \end{theorem}
    
    \begin{proof}
    If $\mathrm{inj}_p(M) > \frac{\pi}{2 \varepsilon}$ then every geodesic in $\Omega_{p,p}$ has length greater than $\frac{\pi}{\varepsilon}$.  From the previous lemma, this shows that every geodesic must have index at least  $(n-1)$.  This implies (see Theorem 32 of  \cite{Petersen}) that $M$ is $(n-1)$-connected.    
    
 An $(n-1)$-connected compact $n$-manifold is homotopic to the sphere.   To see this, first note that by the Hurewicz Theorem we have $\pi_n(M) = H_n(M)$. In the compact case, by Poincare duality,  we have $H_n(M) = \mathbb{Z}$  so  there is a degree 1 map from $M$ to the sphere, and this map induces isomorphisms on $H_i(M)$ for all $i$.  By a theorem of Whitehead (see page 418 of \cite{Hatcher}) a map inducing isomorphisms on all homology groups is a homotopy equivalence.  
 
By a similar topological argument,  an $(n-1)$-connected non-compact $n$-manifold is contractible.  This follows because now we have $H_n(M) = 0$, and a map from a point into $M$ then induces isomorphisms on all homology groups and is thus a homotopy equivalence. 
 
\end{proof}
    
Theorem \ref{Thm:Berg}  shows that  all we need is an injectivity radius estimate  in order to prove Theorem \ref{Thm:Sec}. In the classical case, the injectivity radius estimate under $1/4$-pinching is due to Klingenberg \cite{Kling}.  Our injectivity radius estimate is similar to Klingenberg's, we refer the reader to the texts \cite{doC, Klingbook, Petersen} for the details of the proof.  The first step is  the long homotopy lemma, which only depends on an upper bound on curvature,  so extends immediately to our setting. 
    
    \begin{lemma} [Klingenberg's Long Homotopy Lemma] Suppose that we have $\mathrm{sec} \leq 1$ and suppose that $p, q \in M$ such that $p$ and $q$ are joined by two distinct geodesics  $\gamma_0$ and $\gamma_1$ which are homotopic.  Then there exists a curve in the homotopy $\alpha_{t_0}$ such that 
    \[ \mathrm{length}(\alpha_{t_0}) \geq 2\pi - \mathrm{min}\{ \mathrm{length}(\gamma_i)\} \]
    \end{lemma}
    
The proof of the injectivity radius estimate now follows the same general argument  as in the classical case,  the only difference is that we use Lemma \ref{Lem:Index} to control  the index.  We refer the reader to p. 279 of \cite{doC} for the version of the classical proof on which ours is modeled. 

\begin{theorem} \label{Thm:Kling}  Suppose that $(M,g,X)$ is complete, simply connected,   and has  
 \[  \mathrm{sec}_X \geq \varepsilon > \frac12  \text{ and }\mathrm{sec} \leq 1. \] 
If $p$ is a zero of $X$, then $\mathrm{inj}_p(M) \geq \pi$. 
\end{theorem} 

\begin{proof} 
If $\mathrm{inj}_p(M) < \pi$ then there is  a geodesic loop $\gamma$  based at $p$ of length $l< 2\pi$.     Choose $\delta$ such that
\begin{enumerate}
\item  $\gamma(l- \delta)$ is not conjugate to $p$. 
\item  $\mathrm{exp_p}$ is a diffeomorphism on $B_{p}(2\delta)$. 
\item $3\varepsilon \delta  + N(\delta) < \pi(2\varepsilon - 1)$ where $N(\delta) = \mathrm{max}\{ |X(q)| : q \in B_{p}(2\delta) \}$.
\item $3\delta< 2\pi - l$
\item $5\delta < 2 \pi$, and  
\end{enumerate}

   From Sard's theorem there exists a regular value $q \in B_{\gamma(l-\delta)} (\delta)$ of $\mathrm{exp_p}$.  Since $\gamma(l-\delta)$ is not conjugate to $p$ it is possible to choose a geodesic $\gamma_1$ starting at $p$ to $q$ with $3 \delta < \mathrm{length}(\gamma_1) < l$. 
    
    Let $\gamma_0$ be the minimizing geodesic from $p$ to $q$, then $\mathrm{length}(\gamma_0) \leq 2 \delta$.  In particular, $\gamma_0 \neq \gamma_1$.  Since $M$ is simply connected there is a homotopy $\gamma_t$ from $\gamma_0$ to $\gamma_1$. From the long homotopy  lemma we  have that any such homotopy must contain a curve of length at least     $2\pi - \mathrm{min}\{ \mathrm{length}(\gamma_i)\}$. Since we have $l(\gamma_0) < 2 \delta$ and   $l(\gamma_1)< l <  2\pi - 3 \delta$, the condition $5\delta \leq 2\pi$ implies that we must have a curve of length at least $2\pi - 2\delta$ in the homotopy.  
    
 Now  consider Morse theory applied to the space $\Omega_{p,q}$ with $E$ as Morse function.   Since $q$ is a regular value of $\mathrm{exp_p}$, we have that all of the critical points (geodesics)  are non-degenerate.     From Morse theory  if there was no geodesic with index less than two and length at least  $2\pi - 3 \delta$ then we could push the homotopy from $\gamma_0$ to $\gamma_1$ into the region where $E < 2\pi - 2\delta$ and contradict  the long homotopy lemma.  Therefore, we must have a geodesic from $p$ to $q$ with index zero or one which has length $r$  greater than $2\pi - 3 \delta$.  However, for a parallel field, $E$ along  this geodesic $\sigma$ we have 
 \[ \varepsilon r \leq  \int_0^r \mathrm{sec}_X(\dot{\sigma},  E) dt = \int_0^r \mathrm{sec}(\dot{\sigma},  E) dt + g(X, \dot{\sigma})(q) \leq \int_0^r \mathrm{sec}(\dot{\sigma},  E) dt + N(\delta) \]
 So we have 
 \[ \varepsilon (2\pi - 3 \delta) -  N(\delta)  \leq  \int_0^r \mathrm{sec}(\dot{\sigma},  E) dt\]
 Condition (3) then implies that $ \int_0^r \mathrm{sec}(\dot{\sigma},  E) dt > \pi$ for every unit parallel field along $\sigma$, showing that the index of $\sigma$ is $n-1$, a contradiction to the assumption $\mathrm{inj}_p(M) < \pi$ .   \end{proof}

Combining Theorems \ref{Thm:Berg} and \ref{Thm:Kling} now gives Theorem \ref{Thm:Sec}.   A natural question is whether the $1/2$ pinching constant is optimal in Theorem \ref{Thm:Sec}.  We note that appearance of the $1/2$ as opposed to the $1/4$ in the classical case comes from the difference in the index estimate (Lemma \ref{Lem:Index}).  We know from the examples in section 3 that Lemma \ref{Lem:Index} is optimal, so this at least shows that the method of this paper is not sufficient to improve the pinching constant.  

\end{document}